\documentclass[reqno]{amsart}
\usepackage[margin=1.38in]{geometry}

\pdfoutput=1

\usepackage[foot]{amsaddr}

\usepackage[table,dvipsnames,svgnames,x11names,hyperref]{xcolor}
\usepackage[colorlinks,citecolor=Mahogany,linkcolor=Mahogany,urlcolor=Mahogany,filecolor=Mahogany]{hyperref}
\usepackage{mathtools}
\usepackage{multirow}
\usepackage{tikz}
\usepackage{tikz-cd}
\usepackage{booktabs}
\usepackage[noadjust,nocompress]{cite}
\usepackage[capitalize]{cleveref}
\usepackage[shortlabels]{enumitem}

\newtheorem{theorem}{Theorem}[section]
\newtheorem{lemma}[theorem]{Lemma}
\newtheorem{proposition}[theorem]{Proposition}
\newtheorem{corollary}[theorem]{Corollary}
\theoremstyle{definition}
\newtheorem{definition}[theorem]{Definition}

\theoremstyle{remark}
\newtheorem{example}[theorem]{Example}
\newtheorem{remark}[theorem]{Remark}
\newtheorem*{remark*}{Remark}

\newtheorem{question}[theorem]{Question}

\newcommand{\deff}[1]{\textbf{#1}} %
\DeclareMathOperator{\img}{img}
\newcommand{\Id}{\mathrm{id}}
\newcommand{\free}{\mathcal{L}}
\newcommand{\isomorphic}{\ensuremath{\cong}}
\DeclarePairedDelimiter\abs{\lvert}{\rvert}

\definecolor{darkred}{RGB}{139,0,0}

\newcommand{\countInt}{\mathfrak{I}}
\newcommand{\countSum}{\mathfrak{S}}

\DeclareMathOperator{\Prob}{P}
\DeclareMathOperator{\Exp}{E}

\providecommand\given{}
\newcommand\SetSymbol[1][]{%
  \nonscript\:#1\vert{}
  \allowbreak{}
  \nonscript\:
  \mathopen{}}
\DeclarePairedDelimiterX\Set[1]\{\}{%
  \renewcommand\given{\SetSymbol[]}
  #1
}

\newcommand{\singleton}{\Set{*}}

\newcommand{\bR}{\mathbb{R}}
\newcommand{\cI}{\mathcal{I}}
\newcommand{\cG}{\mathcal{G}}

\newcommand{\cat}[1]{\ensuremath{\mathsf{#1}}}
\newcommand{\cSet}{\cat{Set}}
\newcommand{\cGraph}{\cat{Graph}}
\newcommand{\cVect}{\cat{Vec}}

\newcommand{\myeps}[1]{\rotatebox{45}{$\varepsilon = #1$}}
\newcommand{\mysigma}[1]{\rotatebox{45}{$\sigma = #1$}}
\newcommand{\tinymatrix}[1]{\begingroup\setlength\arraycolsep{2pt}\footnotesize
  \begin{pmatrix} #1\end{pmatrix}\endgroup}
\newcommand{\onetotwo}{\tinymatrix{1 \\ 0}}
\newcommand{\merge}{\tinymatrix{1 & 1}}
\newcommand{\project}{\tinymatrix{1 & 0}}

\newcommand{\realindexing}{\bR_{\geq 0}\times\bR}

\newcommand{\myAcknowledgements}{The authors thank Jan Jendrysiak for helpful discussions. We
  are also grateful to the anonymous reviewers for their careful
  reading of our manuscript and their detailed comments and suggestions.}
\newcommand{\myFunding}{This research has been supported by the Austrian Science Fund (FWF) grant P 33765-N.}

\newif\ifarxiv{}

\arxivtrue{}

\AtBeginDocument{%
	\def\MR#1{}
}

\makeatletter
\renewcommand\subparagraph{\@startsection{subparagraph}{5}%
  \z@{.5\linespacing\@plus.7\linespacing}{-.5em}%
  {\normalfont\bfseries}}
\makeatother

\setenumerate[1]{label=(\roman*),}
\setitemize[1]{label=\raisebox{0.25ex}{\tiny$\bullet$}}

\title{Decomposition of zero-dimensional persistence modules via rooted subsets}

\author[Ángel Javier Alonso]{Ángel Javier Alonso}
\author[Michael Kerber]{Michael Kerber}
\address{Institute of Geometry, Graz University of Technology, Austria.}
\email{alonsohernandez@tugraz.at}
\email{kerber@tugraz.at}

\begin{document}

\ifarxiv{}\else\maketitle\fi

\begin{abstract}
  We study the decomposition of zero-dimensional persistence modules, viewed as
  functors valued in the category of vector spaces factorizing through sets.
  Instead of working directly at the level of vector spaces, we take a step back
  and first study the decomposition problem at the level of sets.

  This approach allows us to define the combinatorial notion of \textit{rooted
    subsets}. In the case of a filtered metric space $M$, rooted subsets relate
  the clustering behavior of the points of $M$ with the decomposition of the
  associated persistence module. In particular, we can identify intervals in
  such a decomposition quickly. In addition, rooted subsets can be understood as a
  generalization of the elder rule, and are also related to the notion of
  constant conqueror of Cai, Kim, Mémoli and Wang. As an application, we give a
  lower bound on the number of intervals that we can expect in the decomposition
  of zero-dimensional persistence modules of a density-Rips filtration in
  Euclidean space: in the limit, and under very general circumstances, we can
  expect that at least 25\% of the indecomposable summands are interval modules.
\end{abstract}

\ifarxiv\maketitle\else\fi

\section{Introduction}
Multiparameter persistent homology is an active research area in topological
data analysis. The motivation is that in many datasets there are multiple
parameters that deserve attention in a multiscale analysis~\cite{carlssonTheoryMultidimensionalPersistence2009,mcinnesAcceleratedHierarchicalDensity2017,buchetTopologicalAnalysisScalar2015}.
Concretely, when analyzing point clouds, we want to consider the distances
between points, but also potentially remove points of low density.

A central object of persistent homology is the \textit{persistence module}, which
tracks algebraically how the topological features of the data change as we move
through the parameter space. In the single-parameter case, every persistence
module decomposes into a collection of intervals, called the \textit{persistence
  barcode}~\cite{collinsBarcodeShapeDescriptor2004}, where each interval
represents the lifetime of a topological feature in the data. In the
multiparameter setting, there is a generalized notion of interval, which again
represents the lifetime of a topological feature, but decomposing a
multiparameter persistence module into intervals is not always possible, and one might be left with non-interval indecomposable persistence modules
that lead to complications, both
theoretically~\cite{carlssonTheoryMultidimensionalPersistence2009,buchetEvery1DPersistence2020,buchetRealizationsIndecomposablePersistence2018,mooreHyperplaneRestrictionsIndecomposable2020}
and computationally~\cite{deyGeneralizedPersistenceAlgorithm2022b,asashibaIntervalDecomposability2D2021,bjerkevikComputingInterleavingDistance2019}.

In fact, the classification of such indecomposable persistence modules is
thought to be out of reach: certain involved posets are of \textit{wild
  representation type}, even when accounting for certain
simplifications~\cite{bauerCotorsionTorsionTriples2020}. Moreover, infinite
families of complicated indecomposable persistence modules can be realized by simple geometric
constructions~\cite{buchetRealizationsIndecomposablePersistence2018}, and, most
recently, it has been shown in~\cite{bauerGenericTwoParameterPersistence2022}
that multiparameter persistence modules are, generically, close to being
indecomposable, under the interleaving metric (we refer
to~\cite{bauerGenericTwoParameterPersistence2022} for a precise statement).

Still, the mentioned complications do not imply that the persistence modules
that come up in practice are close to indecomposable, or that they are not
decomposable into intervals. Indeed, is the decomposition of multiparameter
persistence modules as badly behaved in practice as we can expect in theory? The
authors of~\cite{bauerCotorsionTorsionTriples2020} and those
of~\cite{bauerGenericTwoParameterPersistence2022} state similar questions.

As
an initial test, we computed the decomposition of persistence modules for a
standard zero-dimensional construction, using a prototypical implementation of
the algorithm by Dey and Xin~\cite{deyGeneralizedPersistenceAlgorithm2022b}
(this implementation will be discussed in another paper). As we see
in~\cref{tab:intervals}, the assumption that persistence modules can be
decomposed completely into intervals seems to be false most of the time, at least in this
setting. However, \cref{tab:intervals} also shows that in all tested instances,
\emph{most} indecomposable summands are indeed intervals.

\begin{table*}

\caption{\label{tab:intervals}Number of intervals in the decomposition of zero-dimensional persistence modules for density-Rips filtrations.
    We tried both \textit{clustered} samples where the points were sampled by a multivariate Gaussian distribution around $5$ peaks,
    and \textit{uniform} samples in the unit square.
    The density parameter was computed via a Gaussian kernel density estimate (\textit{kde}) or a \textit{random} density was assigned.
    The table shows the number of intervals for $5$ independent test runs;
    for $n$ points, the module is interval-decomposable if the number of intervals is $n$.
    This only happens for one run.}
\centering
\begin{tabular}[t]{llrrrrrrrrrr}
\toprule
\multicolumn{2}{c}{ } & \multicolumn{5}{c}{100 points} & \multicolumn{5}{c}{500 points} \\
\cmidrule(l{3pt}r{3pt}){3-7} \cmidrule(l{3pt}r{3pt}){8-12}
Sample & Densities & Run 1 & 2 & 3 & 4 & 5 & Run 1 & 2 & 3 & 4 & 5\\
\midrule
\cellcolor{gray!6}{clustered} & \cellcolor{gray!6}{kde} & \cellcolor{gray!6}{100} & \cellcolor{gray!6}{98} & \cellcolor{gray!6}{95} & \cellcolor{gray!6}{98} & \cellcolor{gray!6}{98} & \cellcolor{gray!6}{474} & \cellcolor{gray!6}{487} & \cellcolor{gray!6}{478} & \cellcolor{gray!6}{479} & \cellcolor{gray!6}{479}\\
uniform & kde & 88 & 88 & 86 & 88 & 86 & 444 & 447 & 433 & 453 & 457\\
\cellcolor{gray!6}{clustered} & \cellcolor{gray!6}{random} & \cellcolor{gray!6}{77} & \cellcolor{gray!6}{86} & \cellcolor{gray!6}{87} & \cellcolor{gray!6}{88} & \cellcolor{gray!6}{76} & \cellcolor{gray!6}{397} & \cellcolor{gray!6}{381} & \cellcolor{gray!6}{390} & \cellcolor{gray!6}{380} & \cellcolor{gray!6}{386}\\
uniform & random & 76 & 79 & 75 & 75 & 70 & 376 & 361 & 366 & 355 & 377\\
\bottomrule
\end{tabular}
\end{table*}

This begs the question whether we can provably expect many intervals in general.
In addition, knowledge of the intervals can greatly simplify and speed up
computational tasks for persistence modules: for instance, a popular way to
analyze 2-parameter persistence modules is by considering $1$-dimensional
restrictions, so-called slices, resulting in a parameterized family of
persistent
barcodes~\cite{lesnickInteractiveVisualization2D2015,persistable,rolleStableConsistentDensitybased2021,mcinnesAcceleratedHierarchicalDensity2017}.
Every interval of the $2$-dimensional persistence module gives one bar in the
barcode of the slice, by intersecting the slice with the interval. Thus, by
knowing the intervals, existing algorithms can focus on the non-interval
``core'' of the problem, which is typically of much smaller size.

The practical problem of the described approach is that decomposing a
multiparameter persistence module is costly, despite ongoing
efforts~\cite{deyGeneralizedPersistenceAlgorithm2022b}. However, to leverage the
knowledge of intervals there is no need to compute a total decomposition, or to
even identify all intervals. It suffices to have a method to ``peel off''
intervals from a persistence module quickly. Thus, we pose the question whether there
exist methods that work very fast in practice and still are capable of detecting
many intervals.

\subparagraph{Contributions.}
We focus on the case of zero-dimensional persistence modules.
Already this case is of practical interest because of its connection to
hierarchical clustering methods (see the Related work section below), and has received
attention
recently~\cite{caiElderRuleStaircodesAugmentedMetric2021,bauerCotorsionTorsionTriples2020,rolleStableConsistentDensitybased2021,brodzkiComplexityZerodimensionalMultiparameter2020}.
In this context, we give some answers to the questions stated above:

For a point cloud $M$, a \textit{nearest neighbor pair} is a pair $(x, y)\in M\times M$
such that $y$ is the nearest neighbor of $x$ and $x$ is the nearest neighbor of
$y$ (breaking ties with a fixed total order). The theory we develop says that for a zero-dimensional persistence module of
the density-Rips bifiltration (for any density estimation
function), there are at least as many intervals as there are nearest neighbor
pairs in $M$. These intervals are easily determined by the nearest neighbor
pairs, and we refer to them as \emph{NN-intervals}. Since all nearest neighbor pairs
can be computed in $O(n\log n)$ time~\cite{clarksonFastAlgorithmsAll1983,vaidyaAnOLognAlgorithm1989}, this
yields a fast method to compute all NN-intervals of the decomposition. Moreover,
we can expect many NN-intervals: using previous results on nearest neighbor
graphs, we show that if $M$ is sampled independently from an arbitrary, almost
continuous density function, at least a quarter of the summands
in the decomposition are intervals as $n\to\infty$. To our knowledge, this is
the first result proving a non-constant lower bound on the number of intervals
in a decomposition.

To arrive at this result, we use the following main idea: Instead of studying
the decomposition of the persistence module directly in the category of (graded)
vector spaces, we work in the category of \emph{persistent sets}, whose objects
can be interpreted as a two-parameter hierarchical clustering. The decomposition
of a persistence module is governed by its idempotent endomorphisms, so we look
for idempotent endomorphisms not of persistence modules, but of persistent
sets, which are simpler. We show that such idempotent
endomorphisms can be translated into \emph{rooted subsets}, which are subsets of
points that get consistently merged with a fixed point in the hierarchical
clustering. Moreover, rooted subsets with a single element correspond to
intervals in the associated persistence module.

Instead of peeling off intervals from the persistence module, we peel off rooted
subsets from the persistent set. The advantage is that the remaining structure
is still a hierarchical clustering, and the process can be iterated.

\subparagraph{Related work.} Multiparameter persistent sets and zero-dimensional
persistence modules, as we will study them here, are related to a
multiparametric approach to the clustering problem first considered by Carlsson
and Mémoli~\cite{carlssonMultiparameterHierarchicalClustering2010a}. The need
for multiple parameters, density and scale, is justified by an axiomatic
approach to
clustering~\cite{kleinbergImpossibilityTheoremClustering,carlssonCharacterizationStabilityConvergence2010,carlssonClassifyingClusteringSchemes2013}.
The application of techniques from multiparameter persistence homology, like
persistence modules and interleavings, to this setting has attracted attention
recently~\cite{mcinnesAcceleratedHierarchicalDensity2017,caiElderRuleStaircodesAugmentedMetric2021,bauerCotorsionTorsionTriples2020,rolleStableConsistentDensitybased2021,persistable}.

Cai, Kim, Mémoli, and Wang~\cite{caiElderRuleStaircodesAugmentedMetric2021}
define a useful summary for zero-dimensional persistence modules coming from
density-Rips, called the \textit{elder-rule-staircode}, inspired by the elder
rule~\cite{edelsbrunnerComputationalTopologyIntroduction2010}. They also
introduce the related concept of constant conqueror, and they ask whether a
constant conqueror induces an interval in the decomposition of the associated
persistence module. We answer this question in the negative
with~\cref{ex:counterexample}, and, in contrast, we show that a \textit{rooted
  generator}, as introduced here, does induce an interval in the decomposition
(\cref{thm:main}).

Brodzki, Burfitt, and
Pirashvili~\cite{brodzkiComplexityZerodimensionalMultiparameter2020} also study
the decomposition of zero-dimensional persistence modules. They identify a class
of persistence modules, called \textit{semi-component modules}, that may appear
as summands in the decomposition of zero-dimensional modules, but that are still
hard to classify. Their methods have been of great inspiration, and
in~\cref{thm:lower_bound} we give another proof, within the theory we develop,
of a theorem of theirs.

\ifarxiv{}
\subparagraph{Acknowledgements.} \myAcknowledgements{} \myFunding{}
\fi

\section{Preliminaries}\label{sec:preliminaries}

\subparagraph{Persistent sets and persistence modules.} In what follows, we let
$P$ be a finite poset, which we will view as a category. A
\deff{persistence module} (over $P$) is a functor from $P$ to the category $\cVect$ of
finite dimensional vector spaces, over a fixed field $K$. Such a functor
$F\colon P\to\cVect$ associates to each \deff{grade} $p\in P$ a finite dimensional vector
space $F_{p}$ and to each morphism $p\leq q$ in $P$ a linear map
$F_{p\to q}\colon F_{p}\to F_{q}$, in such a way that $F_{p\to p} = \Id$ and
composition is preserved. We see persistence modules as the objects of the
functor category $\cVect^{P}$, where natural transformations are the morphisms.
In this sense, a morphism $f\colon F\to G$ of persistence modules is a family of
maps $\Set{f_{p}\colon F_{p}\to G_{p}}_{p\in P}$ such that for every two $p\leq q$ the following
diagram commutes
\begin{equation*}
  \begin{tikzcd}
    F_{p} \rar{F_{p\to q}}\dar{f_{p}} & F_{q} \dar{f_{q}} \\
    G_{p} \rar{G_{p\to q}} & G_{q}.
  \end{tikzcd}
\end{equation*}

Similarly, a \deff{persistent set} (over $P$) is a functor from $P$ to $\cSet$, the
category $\cSet$ of finite sets, and morphisms of persistent sets are natural
transformations as above.

We can obtain a persistence module from a persistent set by the application of
the \deff{linearization functor} $\cSet\to\cVect$ that takes each set to the free
vector space generated by it. This linearization functor induces a functor
$\free\colon\cSet^{P}\to\cVect^{P}$ by postcomposition.

\subparagraph{From geometry to persistent sets.} Let $(M, d)$ be a finite metric
space, and consider a function $f\colon M\to\bR$. We can understand $f$ as an
assignment of a \textit{density} to each of the points of $M$; that is, a
density estimation function~\cite{silvermanDensityEstimationStatistics1986}. We
assume that $f$ assigns lower values to points of \textit{higher} density.
Following~\cite{caiElderRuleStaircodesAugmentedMetric2021}, we call the triple
$(M, d, f)$ an \deff{augmented metric space}. We construct a persistent set, the
\deff{density-Rips persistent set} of $(M, d, f)$, that tracks how the
clustering of points of $M$ changes as we change the density and scale
parameters, in a sense that we make precise shortly.

First, for a fixed scale parameter $\varepsilon \geq 0$, we define the
\deff{geometric graph of $M$ at $\varepsilon$}, denoted by
$\cG_{\varepsilon}(M)$, as the undirected graph on the vertex set $M$ and edges
$(x, y)$ where $d(x, y)\leq\varepsilon$. The connected components of
$\cG_{\varepsilon}(M)$, as $\varepsilon$ goes from $0$ to $\infty$, form the
clusters of the dendrogram obtained via the single-linkage clustering method.

To introduce the density, for each $\sigma\in\bR$ we let
$M_{\sigma} \coloneqq \Set{x\in M\given f(x) \leq \sigma} \subseteq M$ be the metric
subspace of points with (co)density below $\sigma$. For any two
$\sigma \leq \sigma'$, $M_{\sigma}\subseteq M_{\sigma'}$ and by taking each
$(\varepsilon, \sigma)$ to the graph $\cG_{\varepsilon}(M_{\sigma})$, we obtain a
functor $\cG(M, f)\colon\realindexing\to\cGraph$, where the order in $\realindexing$ is
given by $(\varepsilon,\sigma)\leq(\varepsilon',\sigma')$ if and only if
$\varepsilon\leq\varepsilon'$ and $\sigma\leq\sigma'$. We then consider the
\deff{connected components functor} $\pi_{0}\colon\cGraph\to\cSet$, that takes each
graph to its set of connected components. In this way, we obtain a functor
$\pi_{0}\circ\cG(M, f)\colon\realindexing\to\cSet$.

\begin{remark}
  The linearized persistence module
  $\free(\pi_{0}\circ\cG(M, f))\colon\realindexing\to\cVect$ is isomorphic to the
  persistence module obtained by applying zero-dimensional homology at graph
  level, $H_{0}\circ\cG(M, f)\colon\realindexing\to\cVect$. In this sense, the
  construction we have described is the zero-dimensional level of the
  density-Rips filtration, which is standard in multiparameter persistent
  homology
  (see~\cite{carlssonTheoryMultidimensionalPersistence2009,blumbergStability2ParameterPersistent2022a}
  and also~\cite{caiElderRuleStaircodesAugmentedMetric2021}).
\end{remark}

We can understand the functor $\pi_{0}\circ\cG(M, f)\colon\realindexing\to\cSet$
as a persistent set $S\colon P\to\cSet$ indexed by a finite grid
$P\subseteq\realindexing$ in the following way. We consider the set of distances
$D \coloneqq \Set{d(x, y)\given x, y\in M}$ and densities
$T\coloneqq\Set{f(x)\given x\in M}$, and define a finite grid
$P\coloneqq D\times T\subset\realindexing$. Finally, we define the persistent set
$S\colon P\to\cSet$ by taking each $(\varepsilon,\sigma)\in P$ to
$(\pi_{0}\circ\cG(M, f))_{(\varepsilon,\sigma)}$, and similarly for the
morphisms.

\begin{definition}\label{def:density_rips}
  Let $(M, d, f)$ be an augmented metric space. We define its \deff{density-Rips
    persistent set} as the functor $S\colon P\to\cSet$, constructed as
  above.
\end{definition}

\subparagraph{Decomposition of persistence modules.} We can study persistence
modules via their decomposition. For two persistence modules $F$ and $G$ their
direct sum $F\oplus G$ is the persistence module given by taking direct sums
pointwise, $(F\oplus G)_{p} = F_{p}\oplus G_{p}$. A persistence module is
\deff{indecomposable} if $F \isomorphic F_{1}\oplus F_{2}$ implies that either
$F_{1} = 0$ or $F_{2} = 0$. Since persistence modules are actual modules (see,
for instance,~\cite[Lemma 2.1]{botnanBottleneckStabilityRank2022}), by the
Krull-Schmidt theorem, a decomposition of a persistence module
$F = F_{1}\oplus F_{2}\oplus\dots F_{n}$ into indecomposable summands is unique
up to permutation and isomorphism of the summands.

Let $I$ be a non-empty connected subposet of a poset $P$ such that for any two
$i,j\in I$ and any $l\in P$, if $i\leq l \leq j$ then $l\in I$. The
\deff{interval module supported on $I$}, $\cI(I)\colon P\to\cVect$, is the
indecomposable (by, e.g.~\cite[Proposition 2.2]{botnanAlgebraicStabilityZigzag2018a}) persistence module
given by
\begin{equation*}
  \cI(I)_{p} =
  \begin{cases}
    K, & \text{if $p\in I$,}\\
    0, & \text{otherwise,}
  \end{cases}
  \quad\text{with internal maps}\quad
  \cI(I)_{p\to q} =
  \begin{cases}
    \Id, & \text{if $p,q\in I$,}\\
    0, & \text{otherwise.}
  \end{cases}
\end{equation*}
If $P$ is a totally ordered set, every persistence module over $P$ decomposes as
a direct sum of interval
modules~\cite{botnanDecompositionPersistenceModules2020}, but such a nice
decomposition does not exist in general for other posets.

\subparagraph{Decomposition and endomorphisms.} A direct sum
$X = X_{1}\oplus X_{2}$ of persistence modules is characterized up to
isomorphism by morphisms $\iota_{i}\colon X_{i}\to X$ and
$\pi_{i}\colon X\to X_{i}$ for $i = 1,2$ such that
$\pi_{i}\circ\iota_{i} = \Id_{X_{i}}$ and
$\iota_{1}\circ\pi_{1} + \iota_{2}\circ\pi_{2} = \Id_{X}$ (see, for
instance,~\cite{maclaneCategoriesWorkingMathematician1978}). In this case, for
each $i = 1, 2$, the maps $\iota_{i}$ and $\pi_{i}$ induce an endomorphism
$\varphi_{i} \coloneqq \iota_{i}\circ\pi_{i}$ of $X$. Such an endomorphism
$\varphi_{i}\colon X\xrightarrow{\pi_{i}} X_{i}\xrightarrow{\iota_{i}} X$ is also split:
\begin{definition}\label{def:split}
  In any category, we say that an endomorphism $\varphi\colon X\to X$ is
  \deff{split} if there exists an object $Y$ and a factorization
  $\varphi\colon X\xrightarrow{\pi} Y\xrightarrow{\iota} X$ such that
  $\pi\circ\iota = \Id_{Y}$.
\end{definition}
We will
use the following standard fact about split endomorphisms\ifarxiv{}\else{}
(proof in the full version)\fi{}:

\begin{lemma}\label{lem:universal_split}
  Let $\varphi\colon X\to X$ be a split endomorphism that has two factorizations
  $X\xrightarrow{\pi} Y \xrightarrow{\iota} X$ and
  $X\xrightarrow{\pi'} Y'\xrightarrow{\iota'} X$ with $\pi\circ\iota = \Id_{Y}$
  and $\pi'\circ\iota' = \Id_{Y'}$. Then $Y$ and $Y'$ are isomorphic.
\end{lemma}
\ifarxiv{}
\begin{proof}
  Consider the compositions
    $\pi'\circ\iota\colon Y \xrightarrow{\iota} X \xrightarrow{\pi'} Y'$ and
    $\pi\circ\iota'\colon Y' \xrightarrow{\iota'} X \xrightarrow{\pi} Y$.
  Using the definition of a split endomorphism we compute
  \begin{equation*}
    (\pi \circ \iota') \circ (\pi' \circ \iota) =
    \pi \circ (\iota' \circ \pi') \circ \iota  =
    \pi\circ\varphi\circ\iota =
    \pi\circ\iota\circ\pi\circ\iota = \Id_{Y},
  \end{equation*}
  and, similarly, $(\pi' \circ \iota) \circ (\pi \circ \iota') = \Id_{Y'}$.
  We conclude that $Y$ and $Y'$ are isomorphic.
\end{proof}
\fi
Every split endomorphism
$\varphi\colon X\xrightarrow{\pi} Y\xrightarrow{\iota} X$ is also
\deff{idempotent}, meaning that $\varphi\circ\varphi = \varphi$. Moreover, in
our categories of interest, namely persistent sets $\cSet^{P}$ and persistence
modules $\cVect^{P}$, every idempotent endomorphism splits through its
\deff{image}, see below. In
these two categories, we define the image of a morphism $f$, $\img f$, by
taking the image pointwise, that is, $(\img f)_{p} = f_{p}(S_{p})$. The
following two lemmas are standard\ifarxiv{}\else{} (proof in the full version)\fi{}.
\begin{lemma}\label{lem:split_image}
  Let $\varphi\colon X\to X$ be an idempotent endomorphism in $\cVect^{P}$ or $\cSet^{P}$.
  Then $f$ splits through its image: there exists a factorization
  $f\colon X\xrightarrow{\pi} \img \varphi \xrightarrow{\iota} X$ with
  $\pi\circ\iota = \Id_{\img \varphi}$.
\end{lemma}
\ifarxiv{}
\begin{proof}
  In any abelian category, like $\cVect^{P}$, every morphism $f\colon X\to Y$
  has a factorization $X\xrightarrow{\pi} \img f \xrightarrow{\iota} Y$ where
  $\pi$ is an epimorphism and $\iota$ a monomorphism (see~\cite[Proposition VIII.3.1]{maclaneCategoriesWorkingMathematician1978}).
  It is easy to check that the same happens in $\cSet$ and $\cSet^{P}$.

  Now, consider an idempotent endomorphism $\varphi\colon X\to X$ and its epi-mono
  factorization $\varphi = \iota\circ\pi$ as above. Since it is idempotent, we have
  $\varphi\circ\varphi = (\iota\circ\pi)\circ(\iota\circ\pi) = \iota\circ\pi$.
  Since $\pi$ is an epimorphism and is right cancellable, and since $\iota$ is a
  monomorphism and is left cancellable, from
  $\iota\circ(\pi\circ\iota)\circ\pi = \iota\circ\pi$ we can obtain
  $\pi\circ\iota = \Id_{\img\varphi}$, as desired.
\end{proof}
\fi{}

\begin{lemma}\label{lem:split_decomposition}
  Let $F\colon P\to\cVect$ be a persistence module, and let
  $\varphi\colon F\to F$ be an idempotent endomorphism. Then $F$ decomposes as
  $\img(\Id_{F} - \varphi)\oplus\img\varphi$.
\end{lemma}
\ifarxiv{}
\begin{proof}
  Applying~\cref{lem:split_image} to $\varphi$, we have a factorization
  $\varphi\colon F\xrightarrow{\pi} \img\varphi \xrightarrow{\iota} F$ with
  $\iota\circ\pi = \varphi$ and $\pi\circ\iota = \Id$. In turn, $\Id-\varphi$ is
  also an idempotent, which splits and satisfies
  $\Id - \varphi + \varphi = \Id$, which, by the characterization of the direct
  sum (the paragraph above~\cref{def:split}), yields a decomposition of the form
  $\img(\Id - \varphi)\oplus\img\varphi$.
\end{proof}
\fi{}

\section{Endomorphisms of persistent sets and rooted subsets}

As seen above, the decomposition of a persistence module is intimately related
to its idempotent endomorphisms. Our main idea is that, when studying the
decomposition of persistence modules of the form $\free S$, for a persistent set
$S\colon P\to\cSet$, we look for idempotent endomorphisms of $S$ and study their image under the
linearization functor $\free$.

\begin{definition}\label{def:generator}
  Given a persistent set $S$, a \deff{generator} is a pair $(p_{x}, x)$ with
  $x\in S_{p_{x}}$ such that $x$ is not in the image of any morphism
  $S_{q\to p_{x}}$ for any $q < p_{x}$. When it is clear, we will often suppress
  the grade $p_{x}$ from the notation, and directly write that $x\in S_{p_{x}}$ is a
  generator.

  There is an induced preorder on the generators of $S$: for two generators
  $x\in S_{p_{x}}$ and $y\in S_{p_{y}}$ we say that $(p_{x}, x) \leq (p_{y}, y)$
  if and only if $p_{x}\leq p_{y}$. This relation might not be antisymmetric,
  and so in general the preordered set of generators is not a poset.
\end{definition}

Generators are useful because an endomorphism $\varphi$ of a
persistent set $S\colon P\to\cSet$ is uniquely determined by the image of its
generators: for each $z\in S_{q}$ we have
$\varphi_{q}(z) = S_{p_{x}\to q}\circ\varphi_{p_{x}}(x)$ for some generator
$x\in S_{p_{x}}$, by
the commutativity property.

In linear algebra, an idempotent endomorphism can be thought as a projection onto
its image, that is, onto its fixed points. This point of view and the concept of
generators above motivates the following
definition, which plays a fundamental role in our work.

\begin{definition}
  A \deff{rooted subset} $A$ is a non-empty subset of the generators of $S$ such
  that there exists an idempotent endomorphism $\varphi$ of $S$ whose set of
  generators that are not fixed is precisely $A$. If a rooted subset is a
  singleton, $A = \Set{x}$, we say that $x$ is a \deff{rooted generator}.
\end{definition}

\begin{remark}
  In the case of an augmented metric space $(M, d, f)$ and its density-Rips
  persistent set $S$ of~\cref{def:density_rips} there exists a bijection between
  the points of $M$ and the generators of $S$. A point $x\in M$ first appears in
  the graph $\cG_{0}(M_{f(x)})$, where $x$ is always its own connected
  component. In what follows, we will often identify a point $x\in M$ with its
  generator $x\in S_{p_{x}}$. In this sense, we can understand an
  endomorphism of $S$ as an endomorphism of the set of points that is compatible
  with the connected components of all graphs $\cG_{\varepsilon}(M_{\sigma})$.
\end{remark}

We are especially interested in persistent sets obtained from (augmented) metric
spaces, and our objective is to relate rooted generators to the
geometry of these objects. Considering an augmented metric space
$(M, d, f)$ and its density-Rips persistent
set,~\cref{prop:criteria_rooted_generator} below characterizes rooted generators
by the clustering behavior of the points of $M$.

\begin{proposition}\label{prop:criteria_rooted_generator}
  Let $(M, d, f)$ be an augmented metric space and consider a point $x\in M$. If
  there exist some other point $y\in M$ such that
  \begin{enumerate}
    \item\label{enum:dense_condition} $f(y)\leq f(x)$ (i.e. $y$ is ``denser''
          than $x$), and
    \item\label{enum:connectedness_condition} whenever $x$ is in a cluster of
          more than one point, $y\in M$ is in the same cluster: for every
          $\cG_{\varepsilon}(M_{\sigma})$, if $x$ is path-connected to some
          other point then $x$ is path-connected to $y$,
  \end{enumerate}
  then the generator $(p_{x}, x)$ of the density-Rips persistent set $S$ of $M$
  is a rooted generator.

  Conversely, if $x$ is a rooted generator of $S$, then there exists a point
  $y\in M$ that satisfies conditions~\ref{enum:dense_condition}
  and~\ref{enum:connectedness_condition} above.
\end{proposition}
\begin{proof}
  Before going into the proof, recall that, by the way we construct $S$
  and the inclusion $P\hookrightarrow\realindexing$, for each $q\in P$ there is an associated graph
  $\cG_{\varepsilon}(M_{\sigma})$, for some $(\varepsilon, \sigma)\in\realindexing$. Each
  element $z\in S_{q}$ is a connected component of this graph
  $\cG_{\varepsilon}(M_{\sigma})$, and the generators $x\in S_{p_{x}}$ such that
  $S_{p_{x}\to q}(x) = z$ are precisely the points in that connected component.

  The first part follows from~\cref{prop:criteria_rooted_subset} below, which
  proves it in more generality.

  For the converse, let $\varphi$ be an idempotent of $S$ whose only generator
  that is not fixed is $x\in S_{p_{x}}$. This means that there exists a
  generator $y\in S_{p_{y}}$, different from $x$, such that
  $\varphi_{p_{x}}(x) = S_{p_{y}\to p_{x}}(y)$. And clearly
  $\varphi_{p_{z}}(z) = z$ for any other generator $z\in S_{p_{z}}$. From the
  fact that $\varphi_{p_{x}}(x) = S_{p_{y}\to p_{x}}(y)$ we deduce that
  $f(y)\leq f(x)$, since $p_{y}\leq p_{x}$ in $P$. To see that the second
  condition holds, pick a $q\geq p_{x}$ and suppose that there exists a
  generator $w\in S_{p_{w}}$ such that $S_{p_{x}\to q}(x) = S_{p_{w}\to q}(w)$. This
  means that in the graph $\cG_{\varepsilon}(M_{\sigma})$ associated to $q$ both
  $x$ and $w$ are in the same connected component, and we claim that $y$ is
  also in this component. Indeed, by the definition of $\varphi$ we have
    $S_{p_{x}\to q}\circ\varphi_{p_{x}}(x) = S_{p_{y}\to q}(y) = S_{p_{w}\to q}(w)$.
\end{proof}

\begin{proposition}\label{prop:criteria_rooted_subset}
  Let $(M, d, f)$ be an augmented metric space. If for a set of points
  $A\subset M$ there exists a point $y\not\in A$ such that for every $x\in A$,
  $f(y)\leq f(x)$ and for each
  $\cG_{\varepsilon}(M_{\sigma})$ either:
  \begin{itemize}
    \item $x$ is path-connected to $y$, or
    \item the set of points that are path-connected to $x$ is contained in $A$,
  \end{itemize}
  then the set of generators $\Set{(p_{x}, x) \given x\in A}$ is a rooted subset
  in the density-Rips persistent set $S$ of $M$.
\end{proposition}
\begin{proof}
  To show that $A$ is a rooted subset, we need to define an appropriate
  idempotent $\varphi$ of $S$. Recalling that an endomorphism is uniquely
  determined by the image of its generators, we define $\varphi$ by setting
  \begin{equation}\label{eq:def_endomorphism}
    \varphi_{p_{x}}(x) =
    \begin{cases}
      S_{p_{y}\to p_{x}}(y), & \text{if $x\in A$},\\
      x, & \text{otherwise},
    \end{cases}
  \end{equation}
  for every generator $x\in S_{p_{x}}$. We need to show that $\varphi$ is indeed
  well-defined, which means that the image of $z\in S_{q}$,
  $\varphi_{q}(z) = S_{p_{x}\to q}\circ\varphi_{p_{x}}(x)$, is the same no
  matter the generator $x\in S_{p_{x}}$ we choose. Fix a $q\in P$ and a
  $z\in S_{q}$, and consider the set $G$ of generators whose image in $S_{q}$ is
  $z$, $G\coloneqq \Set{(p_{x}, x) \given p_{x}\leq q,\ S_{p_{x}\to q}(x) = z}$.

  Then, to check that $\varphi$ is well-defined, for every two
  $(p_{x}, x), (p_{w}, w)\in G$ it must hold that
  \begin{equation}\label{eq:well_defined_condition}
    S_{p_{x}\to q}\circ\varphi_{p_{x}}(x) = S_{p_{w}\to q}\circ\varphi_{p_{w}}(w).
  \end{equation}
  If both $(p_{x}, x)$ and $(p_{w}, w)$ are not in $A$, or if both $(p_{x}, x)$ and $(p_{w}, w)$
  are in $A$,
  then~\cref{eq:well_defined_condition} above trivially holds, by the way we
  have defined $\varphi$ in~\cref{eq:def_endomorphism}.

  Thus, the only interesting case is that only one of $(p_{x}, x)$ or $(p_{w}, w)$ is in $A$. Say that
  $(p_{x}, x)\in A$ and $(p_{w}, w)\not\in A$. Then, by assumption both $x$ and $w$ need to be
  path-connected to $y$ at the graph
  $\cG_{\varepsilon}(M_{\sigma})$ associated to $q$, which means that, as desired,
  \begin{equation*}
    S_{p_{x}\to q}\circ\varphi_{p_{x}}(x) = S_{p_{y}\to q}(y) = S_{p_{w}\to q}(w) = S_{p_{w}\to q}\circ\varphi_{p_{w}}(w).
  \end{equation*}

  Now, $\varphi$ is idempotent, because for every $x\in A$ we have
  $\varphi_{p_{x}}^{2}(x) = \varphi_{p_{x}}(S_{p_{y}\to p_{x}}(y)) = S_{p_{y}\to p_{x}}(\varphi_{p_{y}}(y)) = S_{p_{y}\to p_{x}}(y)$.
  And it is clear that the only generators that are not fixed by $\varphi$ are
  those in $A$. We conclude that, effectively, $A$ is a rooted subset.
\end{proof}

\subparagraph{Decomposition induced by rooted subsets.} As we have seen, rooted
subsets are related to the clustering behavior of the points. They are also
related to the decomposition of the linearized persistence module:
they induce summands.

\begin{theorem}\label{thm:decomp}
  Let $\varphi$ be an idempotent endomorphism of a persistent set $S$. Then the
  persistence module $\free S$ decomposes into
  \begin{equation*}
    \img(\Id_{\free S} - \free\varphi) \oplus \free(\img\varphi).
  \end{equation*}
\end{theorem}
\begin{proof}
  Since $\varphi$ is idempotent, $\free\varphi$ is idempotent.
  By~\cref{lem:split_decomposition}, this induces a decomposition
  \begin{equation*}
    \free S\isomorphic\img(\Id - \free\varphi)\oplus\img\free\varphi.
  \end{equation*}
  It is left to show that $\img\free\varphi\isomorphic\free(\img\varphi)$.
  Applying~\cref{lem:split_image} to $\free\varphi$, we have a factorization
  $\free\varphi\colon\free S\xrightarrow{\pi} \img\free\varphi \xrightarrow{\iota} \free S$
  with $\pi\circ\iota = \Id$. Applying~\cref{lem:split_image} again, this time
  to $\varphi$, we have a factorization
  $\varphi\colon S\xrightarrow{\pi'} \img\varphi \xrightarrow{\iota'} S$ with
  $\pi' \circ \iota' = \Id$. Now, split endomorphisms are preserved by every
  functor: in the diagram
  $\free S \xrightarrow{\free\pi'} \free(\img\varphi) \xrightarrow{\free\iota'} \free S$
  it holds $\free\iota'\circ\free\pi' = \free\varphi$ and
  $\free\pi'\circ\free\iota' = \Id$. Thus, the endomorphism $\free\varphi$ splits in
  two ways:
  \begin{equation*}
    \begin{tikzcd}
      & \img\free\varphi \ar{dd}{\isomorphic} \ar{dr}{\iota} & \\
      \free S \ar{ur}{\pi}\ar{dr}{\free\pi'} & & \free S \\
      & \free(\img\varphi) \ar{ur}{\free\iota'}, &
    \end{tikzcd}
  \end{equation*}
  where the middle arrow exists and is an isomorphism
  by~\cref{lem:universal_split}, finishing the proof.
\end{proof}
Combining the above theorem with the Krull-Schmidt theorem, we obtain the following:
\begin{corollary}\label{thm:main}
  A rooted subset of a persistent set $S$ induces a summand in the
  decomposition of $\free S$.
  A rooted generator $x\in S_{p_{x}}$ induces an interval summand, and all other
  summands can be obtained by decomposing $\free(\img \varphi)$, where $\varphi$
  is the endomorphism associated to $x$.
\end{corollary}

This allows to iteratively \textit{peel off} intervals of a persistence module
of the form $\free S$: find a rooted generator of $S$, with associated
idempotent $\varphi$, and continue considering $\img \varphi$ instead of $S$.
In the setting of an augmented metric space $(M, d, f)$ and its
density-Rips persistent set, the intervals that are peeled off are easily
interpretable through the clustering behavior of the points $M$,
by~\cref{prop:criteria_rooted_generator}. Moreover, the conditions we
describe actually happen in practice, as we see
in~\cref{sec:lower_bound}.

\subparagraph{Neighborly rooted points.} In fact, certain points of an augmented
metric space $(M, d, f)$ can be seen to be rooted by looking at the nearest
neighbors, which will be useful in~\cref{sec:lower_bound}. In what follows we
fix a total order on $M$ compatible with the order induced by $f$. Recall that the
\deff{nearest neighbor} of $x$ is the element $x'\neq x$ of minimum distance to
$x$, where ties have been broken by the fixed total order on $M$.

\begin{definition}\label{def:neighborly_rooted}
  Let $(M, d, f)$ be an augmented metric space. An element $x$ is
  \deff{neighborly rooted} if its nearest neighbor $y\in M$ satisfies
  $f(y)\leq f(x)$.
\end{definition}

\begin{lemma}\label{lem:neighborly_rooted}
  With the notation as above, if a point $x\in M$ is neighborly rooted then $x$ is a rooted
  generator in the density-Rips persistent set of $(M, d, f)$.
\end{lemma}
\begin{proof}
  It is clear that the nearest neighbor of $x$ satisfies the conditions
  of~\cref{prop:criteria_rooted_generator}.
\end{proof}

\begin{remark}
  We can identify all neighborly rooted points in the time it takes to solve the
  all-nearest-neighbor problem.
  Naturally, the all-nearest-neighbor problem can be solved in $O(n^{2})$, where
  $n$ is the number of points, by checking all possible pairs. When the points
  are in Euclidean space, the running time can be improved to $O(n\log n)$ time~\cite{clarksonFastAlgorithmsAll1983,vaidyaAnOLognAlgorithm1989}.
\end{remark}

\subparagraph{Two notable intervals in the decomposition.} The concept of rooted
generators allows us to prove that, in certain cases, we can find at least two
intervals in the decomposition of $\free S$, as in~\cref{thm:at_least_two}
below. We first prove~\cref{thm:bottom_interval}, which has already appeared in
~\cite[Theorem 5.3]{brodzkiComplexityZerodimensionalMultiparameter2020}, where
the proof method is to directly construct an endomorphism of the
persistence module, as we also do after composing with the
linearization functor.

\begin{theorem}\label{thm:bottom_interval}
  Let $S$ be a persistent set. Suppose that the preordered set of generators of
  $S$ has a bottom $\bot$ (that is, one has $\bot\leq x$ for any other generator
  $x$). Then the decomposition of $\free S$ consists of at least one interval,
  induced by $\bot$.
\end{theorem}
\begin{proof}
  Let $\bot\in S_{p_{\bot}}$ be a bottom and let $x\in S_{p_{x}}$ be a generator
  of $S$. Since $\bot\in S_{p_{\bot}}$ is a bottom, we have
  $p_{\bot}\leq p_{x}$. We can define an idempotent $\varphi\colon S\to S$
  by $\varphi_{p_{x}}(x) = S_{p_{\bot}\to p_{x}}(\bot)$ for every generator
  $x\in S_{p_{x}}$ of $S$. This endomorphism is well-defined and its image has
  only one generator, namely $\bot$, and thus $\free(\img\varphi)$ is isomorphic
  to an interval module.
\end{proof}

\begin{theorem}\label{thm:at_least_two}
  Let $(M, d, f)$ be an augmented metric space, and let $S\colon P\to \cSet$ be
  its density-Rips persistent set, as in~\cref{def:density_rips}. If
  $\abs{M} \geq 2$ then the decomposition of $\free S$ into indecomposable
  summands consists of at least two intervals.
\end{theorem}
\begin{proof}
  Consider a point $\top\in M$ of maximal function value, that is,
  $f(\top)\geq f(x)$ for any other $x\in M$.  Let
  $y$ be the nearest neighbor of $\top$. Since $M_{f(\top)} = M_{\sigma}$ for
  any $\sigma \geq f(\top)$, it is clear that $\top$ and its nearest
  neighbor $y$ satisfy the conditions of~\cref{prop:criteria_rooted_generator},
  and thus $\top$ is a rooted generator, yielding the first interval. For the
  second interval, we note that there is at least one point $\bot\in M$ of
  minimal density value and apply~\cref{thm:bottom_interval}.
\end{proof}

\newcommand{\justone}{\begin{array}{cc} x_0 & \\ & \end{array}}
  \newcommand{\allpoints}{\begin{array}{c|c} x_{0} & x_{1}\\ \hline x_{3} & x_{2} \end{array}}
  \newcommand{\allpointsvariant}{\begin{array}{cc}\multicolumn{1}{c|}{x_{0}} & x_{1}\\ \hline x_{3} & x_{2} \end{array}}
  \newcommand{\allpointsbutzero}{\begin{array}{cc}\multicolumn{1}{c|}{x_{0}} & x_{1}\\ \cline{1-1} x_{3} & x_{2} \end{array}}
  \newcommand{\twovertical}{\begin{array}{c|c} x_{0} & x_{1}\\ x_{3} & x_{2} \end{array}}
  \newcommand{\twohorizontal}{\begin{array}{cc} x_{0} & x_{1}\\ \hline x_{3} & x_{2} \end{array}}
  \newcommand{\onlytwo}{\begin{array}{cc} x_0 & x_1
                          \\ \hline \multicolumn{1}{c|}{x_3} & x_2\end{array}}
  \newcommand{\alltogether}{\begin{array}{cc} x_{0} & x_{1}\\ x_{3} & x_{2} \end{array}}
  \newcommand{\firstgraph}{\begin{array}{c|c} x_0 & x_1 \\ & \end{array}}
  \newcommand{\firstgraphsep}{\begin{array}{cc} x_0 & x_1 \\ & \end{array}}
  \newcommand{\threefirst}{\begin{array}{c|c} x_0 & x_1 \\ \cline{2-2} & x_2\end{array}}
  \newcommand{\threesecond}{\begin{array}{c|c} x_0 & x_1 \\ & x_2\end{array}}
  \newcommand{\threethird}{\begin{array}{cc} x_0 & x_1 \\ & x_2\end{array}}
  \newcommand{\threefirstvariant}{\begin{array}{cc} \multicolumn{1}{c|}{x_0} & x_1 \\ \cline{2-2} & x_2\end{array}}
  \newcommand{\graphg}[1]{\includegraphics[page=#1]{ipe/counterexample_all_peeled_just_graphs.pdf}}

\begin{example}\label{ex:not_all_peeled}
  Not every summand of an indecomposable decomposition can be obtained by taking
  rooted subsets and applying~\cref{thm:main}. As an example, consider the
  augmented metric space given by six points $\Set{x_{0}, \dots, x_{5}}$ in the
  plane as in~\cref{fig:not_all_peeled_metric}. Note that $x_{4}$ and $x_{5}$
  are rooted in the associated density-Rips persistent set, and that they can be peeled off. After peeling, we obtain a
  persistent set $S\colon P\to\cSet$ with
  $P \coloneqq \Set{0, 2, 3, 4}\times\Set{0, 1, 2, 3, 4, 5}\subset \bR^{2}$,
  which we
  describe in~\cref{fig:not_all_peeled}. This example is an adaptation
  of~\cite[Example 4.12]{caiElderRuleStaircodesAugmentedMetric2021}, which is
  introduced in the context of \textit{conquerors} that we discuss
  in~\cref{sec:elder_rule}.

  \begin{figure}[h]
    \centering \includegraphics{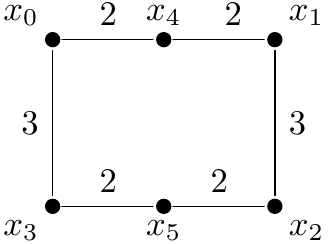}
    \caption{\label{fig:not_all_peeled_metric} The augmented metric space
      $(M, d, f)$ of~\cref{ex:not_all_peeled}, with $f(x_{i}) = i$. These are
      six points $\Set{x_{0}, \dots, x_{5}}$ in the plane, where the distances are given
      by the numbers next to each line.}
    \end{figure}

  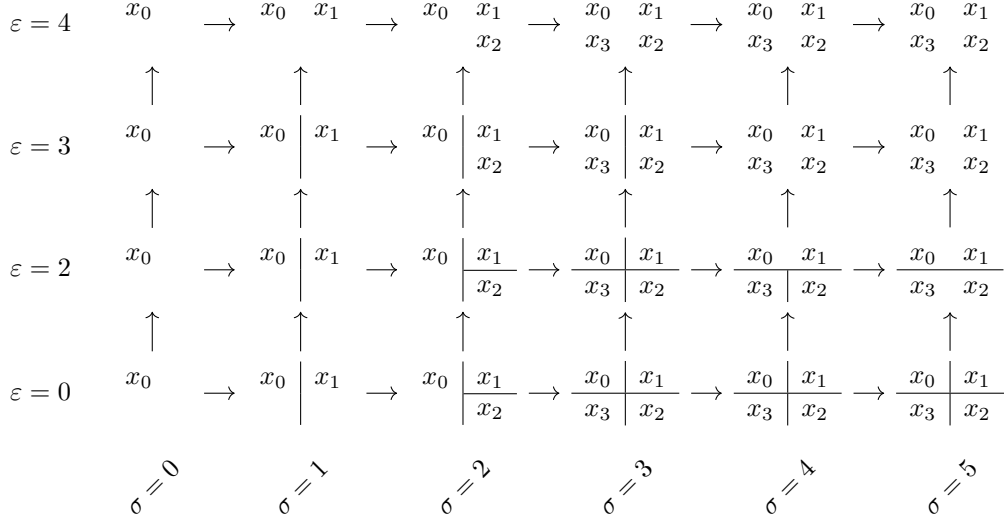
\begin{figure}
    \centering
    \begin{equation*}
      \begin{tikzcd}[ampersand replacement=\&,row sep=15, column sep=12pt]
        \varepsilon  = 4 \&[-5pt] \justone \rar      \& \firstgraphsep \rar             \& \threethird \rar{} \& \alltogether \rar                 \& \alltogether \rar \& \alltogether \\
        \varepsilon  = 3 \& \justone \rar\uar \& \firstgraph \rar\uar \& \threesecond \rar\uar    \& \twovertical \rar\uar             \& \alltogether \rar\uar \& \alltogether \uar \\
        \varepsilon  = 2 \& \justone \rar\uar        \& \firstgraph \rar\uar        \& \threefirst \rar\uar  \& \allpoints \rar\uar \& \onlytwo \rar\uar \& \twohorizontal \uar \\
        \varepsilon  = 0 \& \justone \rar\uar               \& \firstgraph \rar\uar               \& \threefirst \rar\uar         \& \allpoints  \rar\uar              \& \allpoints \rar\uar \& \allpoints \uar \\[-10pt]
        \& \mysigma{0} \& \mysigma{1} \& \mysigma{2} \& \mysigma{3} \& \mysigma{4} \& \mysigma{5}   \&
      \end{tikzcd}
    \end{equation*}
    \caption{\label{fig:not_all_peeled} The persistent set $S\colon P\to\cSet$
      of~\cref{ex:not_all_peeled} obtained by taking the density-Rips
      persistence set of~\cref{fig:not_all_peeled_metric} and removing $x_{4}$
      and $x_{5}$. Each node in the grid represents a partition of the $x_{i}$,
      where $x_{i}$ and $x_{j}$ are in the same partition if they are not
      separated by a line. The arrows are the functions that send the partition
      of $x_{i}$ in one node to the partition of $x_{i}$ in the other.}
  \end{figure}

  We claim that the persistence module $\free S\colon P\to\cVect$ decomposes
  into four summands, all of them interval modules. We denote these summands by
  $I_{0}, I_{1}, I_{2}$ and $I_{3}$, where each $I_{i}$ is associated to the
  generator $(p_{i}, x_{i})$ of $S$, where $p_{i} = (0, i)\in P$. For each
  $i = 0,\dots, 3$, we set $(I_{i})_{p} = 0$ for any $p < p_{i}$ and
  $(I_{i})_{p_{i}} = K$, and we define
  $\iota_{i}\colon I_{i}\to\free S$ by
    \begin{align*}
      (\iota_{0})_{p_{0}} (1) &= [x_{0}], & (\iota_{1})_{p_{1}} (1) &= [x_{1}] - [x_{0}],\\
      (\iota_{2})_{p_{2}} (1) &= [x_{2}] - [x_{1}], & (\iota_{3})_{p_{3}} (1) &= [x_{3}] - [x_{0}] + [x_{1}] - [x_{2}].
    \end{align*}
    The support of each $I_{i}$ are the grades $p\geq p_{i}$ such that
    $((\free S)_{p_{i}\to p}\circ (\iota_{i})_{p_{i}})(1)$ is not zero.
    It can be seen that these maps induce a decomposition
    $\free S \isomorphic I_{0}\oplus I_{1}\oplus I_{2}\oplus I_{3}$.

  However, no subset of the generators other than $\Set{x_{1}, x_{2}, x_{3}}$ is
  rooted because each of the connected
  components given by $\Set{x_{1}, x_{0}}$, $\Set{x_{1}, x_{2}}$,
  $\Set{x_{2}, x_{3}}$, and $\Set{x_{0}, x_{3}}$ appear in $S$.
\end{example}

\section{Rooted generators as a generalization of the elder rule}\label{sec:elder_rule}

\subparagraph{Single-parameter case.} We now suppose that the poset $P$ is a
finite totally ordered poset. In this setting, the theory of rooted generators
allows us to recover the \textit{elder
  rule}~\cite{edelsbrunnerComputationalTopologyIntroduction2010} (see
also~\cite{curryFiberPersistenceMap2018}
and~\cite{caiElderRuleStaircodesAugmentedMetric2021}).

\begin{proposition}
  Let $P$ be a finite totally ordered poset and let $S\colon P\to\cSet$ be a
  persistent set. Suppose that $S$ has at least two generators and that
  $S_{\top}$ is a singleton, where $\top$ is the maximum element of $P$. Then
  every maximal generator (in the preorder of~\cref{def:generator}) is rooted.
\end{proposition}
\begin{proof}
  Let $x\in S_{p_{x}}$ be a maximal generator, and define
  \begin{equation*}
    I_{x} \coloneqq \Set{q\in P \given
      \text{$q\geq p_{x}$ and, for any other generator $w\in S_{p_{w}}$, $S_{p_{w}\to q}(w) \neq S_{p_{x}\to q}(x)$}}.
  \end{equation*}
  Since $p_{x}\in I_{x}$, $I_{x}$ is not empty, and we can consider the set
  $U\subset P$ of upper bounds of $I_{x}$. Moreover, since
  $S_{\top} = \singleton$ and there are at least two generators by assumption,
  the set $U\setminus I_{x}$ is not empty. Let $\alpha$ be the least element in
  $U\setminus I_{x}$. By construction of $I_{x}$ and $U\setminus I_{x}$, there
  is a generator $y\in S_{p_{y}}$ such that
  $S_{p_{y}\to \alpha}(y) = S_{p_{x}\to \alpha}(x)$. Now, since $x$ is maximal,
  it holds that $p_{y}\leq p_{x}$, and we can define an idempotent
  $\varphi\colon S\to S$ by $\varphi_{p_{x}}(x) = S_{p_{y}\to p_{x}}(y)$, and
  $\varphi_{p_{z}}(z) = z$ for any other generator $z\in S_{p_{z}}$. Such an
  idempotent is well-defined by the way we have defined $\alpha$: if there is
  any other generator $w\in S_{p_{w}}$ such that
  $S_{p_{w}\to q}(w) = S_{p_{x}\to q}(x)$ then $\alpha\leq q$ and also
  $S_{p_{y}\to q}(y) = S_{p_{x}\to q}(x)$. We conclude that $x$ is rooted, as
  desired.
\end{proof}

Thus, when $P$ is a total order, we can decompose any persistence module
$\free S$ by peeling off rooted generators, following~\cref{thm:decomp} and
by iteratively considering maximal generators.

\subparagraph{Relation to constant conquerors.} Let $(M, d, f)$ be an augmented
metric space. Cai, Kim, Mémoli and
Wang~\cite{caiElderRuleStaircodesAugmentedMetric2021} define the concept of a
constant conqueror as follows. First, define an ultrametric on $M$: $u(x, x')\coloneqq \min\Set{\varepsilon\in[0, \infty) \given \text{$x$ and $x'$ are path-connected in $\cG_{\varepsilon}(M)$}}$.

Now fix a total order $\prec$ on $M$ and let $x\in M$ be a non-minimal element with
respect to this order. A \deff{conqueror} of $x$ in $M$ is another point $x'\in M$
such that (1) $x' \prec x$, and (2) for any $x''$ with $x'' \prec x$ one has $u(x, x')\leq u(x, x'')$.
Given a function $f\colon M\to\bR$, a \deff{conqueror function} of a non-minimal
$x\in M$, with respect to $\prec$, is a function
$c_{x}\colon [f(x), \infty)\to M$ that sends each $\sigma$ to a
conqueror of $x$ in $M_{\sigma}$. For the minimal element $\bot$ of $M$ we
define $c_{\bot}\colon [f(\bot), \infty)\to M$ to be the constant function at $\bot$.

Also, in the same paper~\cite{caiElderRuleStaircodesAugmentedMetric2021}, given
a point $x\in M$, and assuming that $f\colon M\to\bR$ is injective, the authors
define the \deff{staircode} of $x$ as the set given by
\begin{equation*}
  I_{x} \coloneqq \Set{(\varepsilon, \sigma)\in\realindexing\given x\in M_{\sigma} \text{ and $x$ is the oldest in $[x]_{(\varepsilon,\sigma)}$}},
\end{equation*}
where $[x]_{(\varepsilon, \sigma)}$ is the set of points that are
path-connected to $x$ in $\cG_{\varepsilon}(M_{\sigma})$ and being the
``oldest'' means $f(x) < f(x')$ for any other
$x'\in [x]_{(\varepsilon, \sigma)}$. The authors also define an analogous notion
when $f$ is not injective, which we do not reproduce here.

Finally, the authors ask the following question:
\begin{question}
  Let $(M, d, f)$ be an augmented metric space\@. If $x\in M$ has a constant
  conqueror function, is the interval module supported by $I_{x}$ a summand of
  its density-Rips persistence module?
\end{question}

If we replace constant conqueror by rooted generator then the answer
is yes, by~\cref{thm:main}.
The next example shows that the same cannot hold as originally stated in the
question above.

\begin{example}\label{ex:counterexample}
  Consider the subset $M$ of $\bR$ given by the points $x_{0} = 0$,
  $x_{1} = 7.5$, $x_{2} = 3$ and $x_{3} = 5$. Under the metric induced by the
  Euclidean distance on $\bR$, $M$ is a metric space, and can be made into an
  augmented metric space by defining $f(x_{i}) = i$, see~\cref{fig:counterexample_metric}.
  \begin{figure}[h]
    \centering \includegraphics{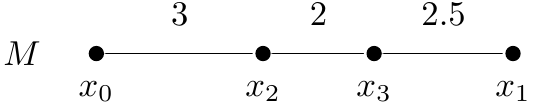}
    \caption{\label{fig:counterexample_metric} The augmented metric space
      $(M, d, f)$ of~\cref{ex:counterexample}, with $M\subset\bR$ and $f(x_{i}) = i$.}
  \end{figure}
  \begin{figure}
    \centering
    \begin{equation*}
      \begin{tikzcd}[ampersand replacement=\&,row sep=12pt, column sep=12pt]
        \sigma  = 3 \&[-5pt] \allpoints \rar      \& \allpointsvariant \rar             \& \allpointsbutzero \rar{} \& \alltogether \rar                 \& \alltogether \rar \& \alltogether \\
        \sigma  = 2 \& \threefirst \rar\uar \& \threefirst \rar\uar \& \threefirst \rar\uar    \& \threefirstvariant \rar\uar             \& \threethird \rar\uar \& \threethird \uar \\
        \sigma  = 1 \& \firstgraph \rar\uar        \& \firstgraph \rar\uar        \& \firstgraph \rar\uar  \& \firstgraph \rar\uar \& \firstgraph \rar\uar \& \firstgraphsep \uar \\
        \sigma  = 0 \& \justone \rar\uar               \& \justone \rar\uar               \& \justone \rar\uar         \& \justone  \rar\uar              \& \justone \rar\uar \& \justone \uar \\[-10pt]
        \& \myeps{0} \& \myeps{2} \& \myeps{2.5} \& \myeps{3} \& \myeps{4.5} \& \myeps{7.5}   \&
      \end{tikzcd}
    \end{equation*}
    \caption{\label{fig:counterexample} We picture the density-Rips persistent
      set of~\cref{fig:counterexample_metric}.}
  \end{figure}
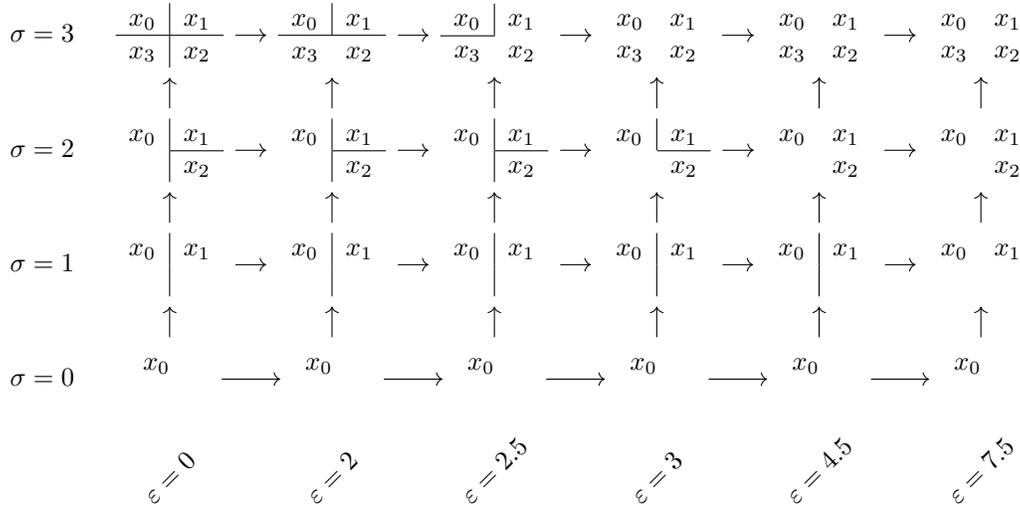

  Consider the only total order $\prec$ on $M$ compatible with $f$,
  $x_{0} \prec x_{1} \prec x_{2} \prec x_{3}$. The point $x_{1}$ has a constant conqueror:
  $x_{0}$ is the only candidate, and it is clear that, for every $i=1,\dots, 3$
  and $x' \prec x_{1}$, $u_{i}(x_{1}, x_{0}) \leq u_{i}(x_{1}, x')$, where $u_{i}$
  is the ultrametric of $M_{i}$, precisely because $x_{0}$ is the only point
  that satisfies $x'\prec x_{1}$.

  Let $S\colon P\to \cSet$ be the density-Rips persistent set constructed from
  the augmented metric space $(M, d, f)$. Here, the poset $P$ is the
  subposet of $\bR^{2}$ given by the grid
  $\Set{0, 2, 2.5, 3, 4.5, 7.5}\times\Set{0, 1, 2, 3}$, where the first
  coordinate represents the distances and the second coordinate the densities.
  We picture $S$ in~\cref{fig:counterexample}.
  Now we proceed to decompose $\free S$. First, note that $x_{3}$ is a rooted
  generator, and consider an associated idempotent $\varphi\colon S\to S$.
  By~\cref{thm:decomp}, there is an interval
  $I\coloneqq \img(\Id_{\free S} - \free\varphi)$ in the decomposition, and we
  can continue considering the persistent set $\img\varphi$. In $\img\varphi$,
  $x_{0}$ is a minimal generator. By~\cref{thm:bottom_interval} (and its proof)
  there is an idempotent $\psi\colon\img\varphi\to\img\varphi$ such that
  $I'\coloneqq\img\free\psi$ is an interval. Applying~\cref{thm:decomp} again,
  we obtain a decomposition of $\free S$ of the form
  \begin{equation*}
    I\oplus I' \oplus \img(\Id_{\free\img\varphi} - \free\psi).
  \end{equation*}
  By direct computation, it can be seen that
  $\img(\Id_{\free\img\varphi} - \free\psi)$ is isomorphic to the persistence
  module described in~\cref{fig:non_interval}. Moreover, this persistence module
  is indecomposable, which can be checked by looking at its endomorphisms: a
  persistence module $F$ is indecomposable if and only if every endomorphism of
  $F$ is either nilpotent or an isomorphism
  (see~\cite{brionRepresentationsQuivers}, and~\cite{brodzkiComplexityZerodimensionalMultiparameter2020}).
  \begin{figure}
    \centering
    \begin{equation*}
      \begin{tikzcd}[ampersand replacement=\&,row sep=30pt, column sep=15pt]
        \sigma  = 3     \& K^{2}   \rar[equal] \&
        K^{2}   \rar{\merge}    \& K \rar{}    \& 0 \rar    \& 0 \rar \& 0 \\
        \sigma  = 2     \& K^{2}   \rar[equal]\uar[equal]    \&
        K^{2}   \rar[equal]\uar[equal]    \& K^{2} \rar{\project}\uar{\merge}    \& K \rar\uar    \& 0 \rar\uar \& 0 \uar \\
        \sigma  = 1     \& K   \rar[equal]\uar{\onetotwo} \& K   \rar[equal]\uar{\onetotwo} \& K \rar[equal]\uar{\onetotwo} \& K \rar[equal]\uar[equal] \& K \rar\uar \& 0 \uar \\
        \sigma  = 0     \& 0   \rar\uar    \& 0   \rar\uar    \& 0 \rar\uar    \& 0 \rar\uar    \& 0 \rar\uar \& 0 \uar \\[-30pt]
        \& \myeps{0} \& \myeps{2} \& \myeps{2.5} \& \myeps{3} \& \myeps{4.5} \& \myeps{7.5}   \&
      \end{tikzcd}
    \end{equation*}
    \caption{\label{fig:non_interval} An indecomposable persistence module
      $F\colon P\to\cVect$, as referenced in~\cref{ex:counterexample}.}
  \end{figure}

  Note that $x_{1}$ is not a rooted generator in $\free S$. In $M_{1}$, $x_{1}$
  is its own connected component during $\varepsilon \in [0, 7.5)$, until
  $x_{0}$ joins the connected component. And in $M_{3}$ it is by itself during
  $\varepsilon\in [0, 2.5)$ and then joins the connected component of $x_{3}$,
  which is not connected to $x_{0}$ at that point. Similarly, $x_{2}$ is not
  rooted.
\end{example}

\begin{remark}
  Note that in Condition (2) of the definition of conqueror, we require that
  $x'' \prec x$. This requirement measures part of the difference between
  constant conqueror function and rooted generator for augmented metric spaces.
  If we drop this requirement, denoting the resulting concept by
  \deff{conqueror$^*$}, we suppose that $f$ is injective, and that $\prec$ is
  compatible with the order induced by $f$, then a non-minimal, with respect to
  $\prec$, point $x\in M$ has a constant \deff{conqueror$^{*}$} function if and
  only if $x$ is a rooted generator, as
  in~\cref{prop:criteria_rooted_generator}.
\end{remark}

\section{A lower bound on the number of expected
  intervals}\label{sec:lower_bound}

We apply the theory we have developed to the study of how a typical
decomposition of a persistence module coming from density-Rips might look like.
In particular, suppose we sample independently $n$ points from a common density
function $f(x)$ in $\bR^{d}$, obtaining a finite metric space
$M\subset \bR^{d}$. We can then consider the augmented metric space $(M, d, f)$,
where $f$, rather than being an estimated density, is the true underlying
density function. This setting resembles actual practice, but is more suitable
to theoretical study. Let $S$ be the density-Rips persistent set of $M$. Then,
how many intervals can we expect in the decomposition of $\free S$? The
following theorem says that, under very general conditions on $f$, regardless of
$d$, and as $n$ goes to infinity, we can at least expect $25\%$ of the summands
to be intervals.

\begin{theorem}\label{thm:lower_bound}
  Let $X_{1}, \dots, X_{n}$ be i.i.d.\ points taking values in $\bR^{d}$,
  sampled from a common density function $f(x)$ that is continuous almost
  everywhere with respect to the Lebesgue measure.

  Consider the finite augmented metric space
  $(M = \Set{X_{1}, \dots, X_{n}}, d_{M}, f)$, where $d_{M}$ is induced by the Euclidean
  metric in $\bR^{d}$, and let $S$ be its density-Rips persistent set.

  Let $\countInt_{n}$ be the random variable that counts the number of intervals
  in the indecomposable decomposition of $\free S$, and let $\countSum_{n}$ be
  the random variable that counts the total number of summands in the same
  decomposition. We have
  \begin{equation}\label{eq:lower_bound}
    \liminf_{n\to\infty} \Exp\left[\frac{\countInt_{n}}{\countSum_{n}}\right] \geq c(d),
  \end{equation}
  where $c(d)$ is a constant that depends on $d$, and $c(1) = \frac{1}{3}$,
  $c(2) \approx 0.31$ and $c(d)\downarrow\frac{1}{4}$ as $d\to\infty$.
\end{theorem}
The rest of the section is dedicated to proving this theorem. The nearest
neighbor graph of a metric space plays a fundamental role.

\begin{definition}
  The \deff{nearest neighbor graph} of $M$ is the directed graph on $M$ given by
  the directed edges of the form $(x, x')$, where $x'$ is the nearest neighbor
  of $x$.
\end{definition}

Now, we are interested in estimating the number of neighborly rooted elements,
as in~\cref{def:neighborly_rooted}, as they induce an interval in the
decomposition of $\free S$. However, in general being neighborly rooted depends
on $f$. To do without the condition on $f$ we have:
\begin{lemma}\label{lem:weakly_conn_components}
  Let $(M, d_{M}, f)$ be an augmented metric space and let $S$ be its density-Rips
  persistent set. There are at least as many intervals in the indecomposable
  decomposition of $\free S$ as $2$-cycles in the nearest neighbor graph of $M$.
\end{lemma}
\begin{proof}
  We can assume without loss of generality that $\abs{M} \geq 2$. Let $G$ be the
  nearest neighbor graph of $M$. The only cycles in this graph are precisely the
  2-cycles, and each weakly connected component of $G$ contains exactly one
  2-cycle (see~\cite{eppsteinNearestNeighborGraphs}).

  Let $C_{1}, \dots, C_{k}$ be the weakly connected components of $G$. Fix
  $i \in \Set{1, \dots, k}$, and let $x, y\in M$ be such that $(x, y)$ and
  $(y, x)$ is the 2-cycle in $C_{i}$. Either $f(y)\leq f(x)$ or $f(x)\leq f(y)$,
  and either $x$ is neighborly rooted, $y$ is neighborly rooted, or both are
  neighborly rooted. Say $x$ is neighborly rooted, and define an endomorphism
  $\varphi_{i}\colon S\to S$ by setting
  \begin{equation*}
    (\varphi_{i})_{p_{z}}(z) =
    \begin{cases}
      S_{p_{y}\to p_{x}}(y), & \text{if $x = z$,}\\
      z, & \text{otherwise,}
    \end{cases}
  \end{equation*}
  for every generator $z\in S_{p_{z}}$. Such an endomorphism is well-defined as
  shown in~\cref{prop:criteria_rooted_generator}.

  Constructing, for each $i$, an idempotent $\varphi_{i}$ as above, it is clear
  that we can iteratively peel off the associated intervals, yielding the desired
  conclusion.
\end{proof}

Naturally, the number of 2-cycles is half the number of points that are the
nearest neighbor of its nearest neighbor. The problem of estimating the
probability for a point to be the nearest neighbor of its nearest neighbor,
assuming a random point process, has been studied by multiple authors
(see~\cite{schillingMutualSharedNeighbor1986,henzeProbabilityThatRandom1986,henzeFractionRandomPoints1987,coxReflexiveNearestNeighbours1981,eppsteinNearestNeighborGraphs}).

In our case, when we have $X_{1}, \dots, X_{n}$ i.i.d.\ points in $\bR^{d}$
sampled from a common density function $f$ under the conditions
of~\cref{thm:lower_bound}, by~\cite[Theorem 1.1]{henzeFractionRandomPoints1987},
and letting $N_{i,n}$ denote the probability event that $X_{i}$ is the nearest
neighbor of its nearest neighbor, we have
\begin{equation}\label{eq:prob_nn}
  \lim_{n\to\infty}\Prob(N_{i,n}) = b(d),
\end{equation}
where $b(d)$ is the volume of a unit $d$-sphere divided by the volume of the
union of two unit spheres with centers at distance $1$. In fact,
$b(1) = \frac{2}{3}$, $b(2) \approx 0.621$, and $b(d)\downarrow \frac{1}{2}$ as
$d\to\infty$ (see~\cite[Table 2]{schillingMutualSharedNeighbor1986}), and we
define $c(d)\coloneqq \frac{b(d)}{2}$.

We are now ready to finish the proof of~\cref{thm:lower_bound} at the start of
the section. Applying~\cref{lem:weakly_conn_components} and the linearity of
expectation, it holds
\begin{equation*}
  \Exp[\countInt_{n}] \geq
  \Exp\left[\sum_{i=1}^{n}\frac{I(N_{i,n})}{2}\right] =
  \sum_{i=1}^{n}\frac{\Exp[I(N_{i,n})]}{2} = \sum_{i=1}^{n}\frac{\Prob(N_{i,n})}{2},
\end{equation*}
where $I(N_{i})$ is the indicator random variable of $N_{i,n}$.
By~\cref{eq:prob_nn} we have
\begin{equation*}
  \liminf_{n\to\infty} \Exp\left[\frac{\countInt_{n}}{n}\right] \geq \frac{b(d)}{2} = c(d).
\end{equation*}
Finally, noting that the number of summands in the decomposition is bounded by
the number of points, $\countSum_{n}\leq n$\ifarxiv{}, by the lemma
below\else{} (see full version)\fi{}, \cref{eq:lower_bound}
of~\cref{thm:lower_bound} follows, finishing the proof.

\ifarxiv{}
\begin{lemma}
  Let $(M, d_{M}, f)$ be an augmented metric space, and let $S\colon P\to\cSet$ be
  its density-Rips persistent set. Then any decomposition of $\free S$ consists
  of at most $n = \abs{M}$ summands.
\end{lemma}
\begin{proof}
  Given a persistence module $F\colon P\to\cVect$, denote by
  $\beta_{0}(F)\colon P\to\mathbb{N}$ the function that assigns to each $p\in P$
  the multigraded $0$-th Betti number of $F$ at $p\in P$ (we refer
  to~\cite{lesnickComputingMinimalPresentations2022} for their definition).
  Precisely because $S$ has $n$ generators, it is not hard to see that
  $\sum_{p\in P}\beta_{0}(\free S)(p) = n$.

  Consider a decomposition $\free S\isomorphic X_{1}\oplus\cdots\oplus X_{k}$.
  Since Betti numbers are additive, we have
  $\beta_{0}(\free S) = \sum_{i=1}^{k}\beta_{0}(X_{i})$. Therefore, we can write
  \begin{equation}
   \sum_{p\in P} \beta_{0}(\free S)(p) = \sum_{i=1}^{k}\sum_{p\in P}\beta_{0}(X_{i})(p) = n.
  \end{equation}
  Now, since $\sum_{p\in P}\beta(X_{i})(p) \geq 1$ if $X_{i}\neq 0$, it follows
  that $k\leq n$.
\end{proof}
\fi{}

\section{Discussion}

Although we have focused our attention to augmented
metric spaces and density-Rips, rooted subsets can be applied to other
persistent sets. Of special interest for us is the degree-Rips
filtration~\cite{blumbergStability2ParameterPersistent2022a} of a metric space,
where we filter by the degree of the vertices in the underlying geometric
graphs. To accommodate this situation, one could modify condition 1
of~\cref{prop:criteria_rooted_generator} to take into account the evolution of
the degrees, rather than the density. We leave an in-depth treatment of this
case for future work.

We have seen, both in our lower bound of~\cref{sec:lower_bound} and in
preliminary experimental evaluation, that we can expect to find many intervals
in the decomposition of those persistence modules coming from geometry, at least
in the cases considered here. This is in contrast to the purely algebraic
setting, where, in light of recent
developments~\cite{bauerCotorsionTorsionTriples2020,bauerGenericTwoParameterPersistence2022},
looking for a decomposition might fall short.

\bibliographystyle{amsplainurl}{}
\bibliography{refs}

\end{document}